\documentclass{article}
\usepackage{graphicx}

\usepackage[fleqn]{amsmath}

\usepackage{amsthm}
\usepackage{amssymb}
\usepackage{amsbsy}
\usepackage{amsfonts}
\usepackage{mathrsfs}
\usepackage{amsmath}
\usepackage[all]{xy}
\usepackage{amstext}
\usepackage{amscd}
\usepackage[dvips]{epsfig}
\usepackage{psfrag}
\usepackage{enumerate}
\usepackage{flafter}
\allowdisplaybreaks

\usepackage{tikz}

\theoremstyle{plain}
\newtheorem{thm}{Theorem}[section]
\newtheorem{prop}[thm]{Proposition}
\newtheorem{cor}[thm]{Corollary}
\newtheorem{lem}[thm]{Lemma}
\theoremstyle{definition}
\newtheorem{exa}[thm]{Example}

\newtheorem{rem}[thm]{Remark}

\def\Ker{\mathop{\mathrm{Ker}}\nolimits}

\newcommand{\lra}{\longrightarrow}
\newcommand{\ra}{\rightarrow}

\newcommand{\R}{{\Bbb R}}
\newcommand{\Z}{{\Bbb Z}}

\begin{document}
\large
\begin{center}{\bf\Large Cocycles of nilpotent quotients of free groups}
\end{center}
\vskip 1.5pc
\begin{center}{\Large
Takefumi Nosaka\footnote{E-mail address: {\tt nosaka@math.titech.ac.jp}}}\end{center}
\vskip 1pc
\begin{abstract}\baselineskip=12pt \noindent

We focus on the cohomology of the $k$-th nilpotent quotient of the free group, $F/F_k$.
This paper describes all the group 2-, 3-cocycles in terms of Massey product,
and gives expressions for some of the 3-cocycles.
We also give simple proofs of some of the results on Milnor invariant and the Johnson-Morita homomorphisms.
\end{abstract}
\begin{center}
\end{center}

\tableofcontents
\large
\baselineskip=15.5pt

\section{Introduction}
Let $F$ be the free group of rank $q$. We define $F_1$ to be $F$, and $F_k$ to be the commutator subgroup $ [F_{k-1} ,F ]$ by induction. Accordingly, we have the central extension,
\begin{equation}\label{kihon2} 0 \lra F_k/F_{k+1} \lra F/F_{k+1}  \xrightarrow{\ \ p_{k}\ \ }F/F_{k} \lra 1 \ \ \ \ \ (\mathrm{central \ extension})\end{equation}
The abelian kernel $F_k/F_{k+1} $ is classically known to be free and of finite rank (see, e.g., \cite{BC,MKS} and references therein). We denote the rank by $N_k \in \mathbb{N}$.

The nilpotent quotient $F/F_{k} $ have been studied with applications; see, e.g., \cite{Mas,BC,MKS} for the relation to the free Lie algebra $ \oplus_{i=1}^k F_k/F_{k+1}$. The group homology of $F/F_k$ also plays a fruitful role in the study of low dimensional topology, including the Milnor (link) invariant, the Johnson-Morita homomorphisms, and tree parts of the quantum invariant; see \cite{GL,Ki,IO,Heap,KN,Mas,Tur,Por}. Massey products of manifolds appear in nilpotent obstructions of manifolds, as described in \cite{CGO,FS,Ki,GL,Heap}. Moreover, 
the homology groups of degree 2 and 3 are computed as
\begin{equation}\label{kihon7} H_2( F/ F_k;\Z ) \cong \Z^{ N_k} , \ \ \ \ \ \ H_3( F/ F_k;\Z ) \cong \bigoplus_{i=k}^{2k-2}\Z^{ q N_i -N_{i+1}}.
\end{equation}
The former is a classical result; however, the latter was a result of Igusa and Orr \cite[Corollaries 5.5 and 6.5]{IO} shown by spectral sequences. In particular, 
it seems hard to deal with the third homology $H_3( F/ F_k;\Z )$ quantitatively.

In this paper, we describe bases of the cohomology $H^2( F/ F_k;\Z ) $ and $H^3( F/ F_k;\Z ) $ as Massey products (see Theorem \ref{maainthm}) and explicitly express some of the cocycles. 
Furthermore, we consider the Massey products to be an algorithm to produce many expressions of cocycles. In fact, Section \ref{So} gives 3-cocycles of some nilpotent groups, and concretely describes their expressions.
In doing so, it is reasonable to hope that the expressions may be useful for computing various things appearing in topology; including the Morita homomorphisms \cite{Morita2} and the Orr link invariants \cite{Orr}.
Incidentally, the geometric realization of \eqref{kihon2} is the iterated torus bundle; thus, in Appendix A, we can express the 2-cocycles as differential 2-forms in the de Rham complex; see Theorem \ref{32az}.

As corollaries, Section \ref{Sec3} provides simple proofs of four known results of \cite{FS,Por,Tur,Heap,Ki}, which are related to Massey products. The original proofs were discussed at the cohomology level, and they constructed Massey products according to circumstances. In contrast, since the above theorem ensures cocycles which gives a basis of $H^*( F/ F_k;\Z ) $ as Massey products, we will give shorter proofs in terms of cocycles.

This paper is organized as follows. Section 2 reviews Magnus expansions and Massey products, and Section 3 states our theorem with a proof. Section 4 explains the 
four known results mentioned above and alternative proofs based on our theorem in Section 3. Section 5 discusses an algorithm to produce cocycles.

\

\noindent
{\bf Conventional notation}. \ Throughout this paper, let $F$ be the free group of rank $q$, and let $N_k \in \Z$ be the rank of $F_k/F_{k+1}$. Given a group $G$, we denote $G$ by $G_1$ and define $G_k$ to be the commutator subgroup $ [G_{k-1} ,G ]$ by induction. For a group $G$, we write $BG$ for the Eilenberg-MacLane space, i.e., $K(G,1)$-space. Furthermore, we assume the basic properties of group (co)-homology as in \cite[Sections I, II, and VII]{Bro}.

\section{Review: Magnus expansions and higher Massey products}
\label{Semddb}
Let us begin by studying unipotent Magnus expansions and higher Massey products.

First, we review the Magnus expansion modulo degree $k$. Let $\Z \langle X_1,\dots, X_q\rangle $ be the polynomial ring with non-commutative indeterminates $X_1,\dots, X_q $, and $\mathcal{J}_{k}$ be the two-sided ideal generated by polynomials of degree $\geq k $. Then, {\it the Magnus expansion (of the free group $F$)} is the multiplicative map $ \mathcal{M}: F \ra \Z \langle X_1,\dots, X_q \rangle /\mathcal{J}_{k} $ defined by
\begin{equation}\label{mag} \mathcal{M}(x_i) = 1+X_i, \ \ \ \ \ \ \ \mathcal{M}(x_i^{-1}) = 1- X_i+X_i^2 + \cdots + (-1)^{k-1}X_i^{k-1}.
\end{equation}
As is known, $\mathcal{M} (F_k ) =0 $. By passage to this $F_k$, this $\mathcal{M}$ further induces the injection
$$ \mathcal{M}:F/ F_k \lra \Z \langle X_1,\dots, X_k \rangle /\mathcal{J}_{k}. $$

Next, we review another description of the Magnus embedding \cite{GG}, which is a faithful linear representation of $ F/F_k $. Let $\Omega_k $ be the polynomial ring $\Z[\lambda_{i}^{(j)}]$ over commuting indeterminates $ \lambda_{i}^{(j)}$ with $ i \in \{ 1,2, \dots, k-1 \}, \ j \in \{ 1, \dots, q\}$. We define the group homomorphism 
$$\Upsilon_k : F \lra GL_k(\Omega_k ), $$
by setting
$$ \Upsilon_k(x_j)= \left(
\begin{array}{cccccc}
1 & \lambda_{1}^{(j)} & 0& \cdots & 0\\
0 & 1& \lambda_{2}^{(j)} & \cdots & 0\\
\vdots & \vdots & \ddots & \ddots & \vdots \\
0&0 & \cdots & 1 & \lambda_{k-1}^{(j)} \\
0& 0& \cdots & 0 & 1
\end{array}
\right).
$$
As is known \cite{GG}, the image $\Upsilon_k ( F_k)$ consists of the identity matrix, and the quotient map $ F/F_{k} \ra GL_k( \Omega_k ) $ is injective. Thus, we have an isomorphism $F/F_{k} \cong \mathrm{Im}(\Upsilon_k). $ It is shown \cite[Appendix]{KN} that the correspondence $1+ X_j \mapsto \Upsilon_k(x_j) $ gives the equivalence beteween  the former $\mathcal{M} $ and $\Upsilon_k $. The longer the word of $x$ is, the harder the computation of $\mathcal{M} (x)$ is; however, that of $\Upsilon_k (x)$ is still simpler. Indeed, the map $\Upsilon_k $ is defined over the commutative ring $\Omega_k $, and is therefore compatible with computer programs.

Next, we review the higher Massey products, which were first defined by Kraines \cite{Kra}. Here, we describe the products in the non-homogenous complex of a group $G$ with a trivial coefficient ring $A$. That is, as in \cite[Chap. III.1]{Bro}, we define the group $C^{*} (G;A )$ of cochains to be $\mathrm{Map}(G^n,A)$ and the coboundary map $\partial_n^*$ by setting
\[(\partial_n^* f ) ( g_1, \dots ,g_{n}) = f( g_2, \dots, g_n )+ (-1)^n f ( g_1, \dots ,g_{n-1}) \]
\[ \ \ \ \ \ - f (g_1 g_2 , g_3, \dots, g_n) + f (g_1 ,g_2 g_3, g_4 \dots, g_n) + \cdots + (-1)^{n-1} f ( g_1, \dots ,g_{n-1}g_n) . \]
Furthermore, the cup product on $C^n (G;A) $ can be described as a canonical product. More precisely, for $u \in C^p( G;A )$ and $v \in C^{q}( G;A )$, the product $u \smile v \in C^{p+q}( G;A )$ is defined by
$$ ( u \smile v) ( g_1 ,\dots, g_{p+q}):= (-1)^{pq} u (g_1 ,\dots, g_{p} ) \cdot v (g_{p+1} ,\dots, g_{p+q} ) \in A .$$

For any $1 \leq i \leq n$, take a cocycle $ \gamma_i \in C^{p_i} (G ;A )$. Then, a {\it defining system} associated with $(\gamma_1,\dots, \gamma_n)$ is a set of elements $(a_{s,t})$ for $1 \leq s \leq t\leq n$ with $(s,t) \neq (1, n)$, satisfying

\begin{enumerate}[(i)]
\item $a_{s,t} \in C^{p_s +p_{s+1}+ \cdots + p_{t} -t+s} (G;A ) $.
\item When $s=t$, the diagonal map $a_{s,s}$ is equal to $ \gamma_s$ in $C^{p_s} (G;A )$.
\item $ \partial^* ( {a_{s,t}})= \sum_{r=s}^{t-1} (-1)^{p_s +p_{t+1}+ \cdots + p_{t} -t+s} a_{s,r}\smile a_{r+1,t}.$
\end{enumerate}
Given such a defining system, we can define a cocycle of the form,
\begin{equation}\label{kihon4} \sum_{r : \ 1 \leq r \leq n-1} (-1)^{p_1 +p_{2}+ \cdots + p_{r} -r+1} a_{1,r}\smile a_{r+1,n} \in
C^{p_1 +p_{2}+ \cdots + p_{n } -n+2} (G;A ) .
\end{equation}
Following \cite{Kra}, the {\it $n$-fold Massey product}, $\langle \gamma_1, \gamma_2,\dots, \gamma_n\rangle $, is defined to be the set of cohomology classes of cocycles associated with all possible defining systems. While there are many interpretations of the higher Massey product, this paper uses the Massey products as a method to yield cocycles from other cocycles of lower degree.

\begin{rem}\label{pp2}
As is known \cite{FS}, if $p_1=p_2 = \cdots =p_n=1 $ and every $m$-fold Massey products with $m<n$ chosen from $ \{\gamma_1, \dots, \gamma_n \}$ is null-cohomologous, the {\it $n$-fold Massey product} $\langle \gamma_1, \gamma_2,\dots, \gamma_n\rangle $ is a singleton in $H^2(G;A)$.
\end{rem}

\section{Main theorem: generators of $H^* (F/F_k) $}
\label{Sbra1}
We will describe bases of some cohomology of $F/F_k$ in terms of Massey products (Theorem \ref{maainthm}).

Before stating the theorem, we should concretely define some defining systems. Let
$F$ be the free group with basis, $x_1, \dots, x_q$,
and $G$ be $ F/F_k. $ For $ 1 \leq t \leq q$,
let $\alpha_{t} : F/F_k \ra \Z$ be the homomorphism which sends $x_s$ to $\delta_{s,t} $.
We regard $\alpha_t$ as a 1-cocycle of $ F/F_k $. Then, given a $k$-tuple $I=(i_1 , \dots, i_{k}) \in \{ 1,2, \dots, q\}^{k} $, we have 1-cocycles $ \alpha_{i_1} ,\dots, \alpha_{i_k} $. Furthermore, for $1 \leq s \leq t\leq k$, let us consider the evaluation of the coefficient of $X_{i_s} \cdots X_{i_t}. $ That is, we set up the linear map,
\begin{equation}\label{mag1231} \beta_{i_s i_{s+1} \cdots i_t}: \Z \langle X_1 , \dots, X_q \rangle \lra \Z; \ \ \ \ \ \ \sum a_{ j_1 \dots j_a} X_{j_1} \cdots X_{j_a} \longmapsto a_{ i_s i_{s+1}\dots, i_t}.
\end{equation}
Let us denote the composite $ \beta_{i_s \cdots i_t} \circ \mathcal{M}$ by $c_{i_s \cdots i_t}$, which we will use many times.

\begin{lem}\label{maainlem}
Consider the case $p_1=p_2 = \cdots =p_n=1 $ with $n=k $. Let $a_{s,t}: F \ra \Z$ be the composite $c_{i_s \cdots i_t} $. Then, the set of $(a_{s,t})$ is a defining system associated with $(\alpha_{i_1} , \dots, \alpha_{i_k} )$. In particular, the resulting 2-cocycle is represented by
\begin{equation}\label{mag139} F/F_k \times F/F_k \lra \Z ; \ \ \ (x,y) \longmapsto \sum_{\ell: \ 1\leq \ell \leq k-1} c_{i_1 i_{2}\cdots i_\ell} (x) c_{i_{\ell+1} \cdots i_k} (y) .
\end{equation}
\end{lem}

\begin{proof}
From the unipotent Magnus expansion, the right hand side in (iii) is equivalent to the product of upper triangular matrices. Thus, it is not so hard to check (iii) by direct computation. Thus, the formula of the Massey products \eqref{kihon4} readily means \eqref{mag139}.
\end{proof}

Moreover, let us review standard sequences from \cite{CFL}. Equip the set of all sequences $\bigcup_{s=1}^{\infty} \{ 1, 2, \dots, q \}^s$ with the lexicographical order. Then, a sequence $I= i_1 i_2 \cdots i_k$ is said to be {\it standard}, if
$ I < i_s i_{s+1} \cdots i_k $ for any $ 2 \leq s \leq k$.
For example, $\{ 123\}$ and $\{ 1223\}$ are standard, but $\{ 213\}$ and $\{ 3142\}$ are not standard.
Let $\mathfrak{U}_k$ be the set of standard sequences of length $k$. As is known (see, e.g., \cite{MKS} and \cite[Theorem 1.5]{CFL}), the order of $ \mathfrak{U}_k$ is equal to $N_k.$ Furthermore, the following is known:
\begin{equation}\label{kihon3} N_k:= \mathrm{rank} (F_k/F_{k+1})= \mathrm{rank}\bigl( H^2(F/F_k;\Z) \bigr) =| \mathfrak{U}_k|= \frac{1}{k}\sum_{d: \ d|k} \mu (\frac{k}{d})q^d \in \mathbb{N}, \end{equation}
where $\mu$ is the M\"{o}bius function (see, e.g., \cite{Witt} or \cite[Theorem 1.5]{CFL}), and the last equality is commonly called Witt's formula.

While $F_k/F_{k+1} $ is classically known to be spanned by Hall basis (see, e.g., \cite{Hall,MKS})
we will give a basis of $H^2(F/F_k;\Z) $ for applications to the homology.
%
Precisely, the second and third cohomology of $ F/F_k $ are generated by Massey products:
\begin{thm}\label{maainthm}
\begin{enumerate}[(I)]
\item Every $ j$-fold Massey product with $j <k$ is zero. In particular, for any standard index $i_1 \cdots i_k \in \mathfrak{U}_k $, the $k$-fold one $ \langle \alpha_{i_1}, \dots, \alpha_{i_k} \rangle $ is uniquely defined in $H^2( F/F_k ;\Z)$ and is represented by a 2-cocycle in \eqref{mag139}.
\item The second cohomology $H^2( F/F_k ;\Z) \cong \Z^{N_k} $ is spanned by the $k$-fold Massey products $ \langle \alpha_{i_1}, \dots, \alpha_{i_k} \rangle $ running over all the standard sequences $(i_1 \cdots i_k ) \in \mathfrak{U}_k$.

\item For any $k \leq \ell \leq 2k-2$, consider the projection $p_{\ell}: F/F_{\ell} \ra F/F_k $. Then, there are homomorphisms $ \mathfrak{s}_{\ell} : Z^3( F/F_{\ell} ;\Z) \ra Z^3( F/F_k ;\Z) $ such that
\begin{equation}\label{bbbbbb} \alpha_{s} \smile \langle \alpha_{i_1}, \dots, \alpha_{i_{\ell}} \rangle = p^*_{\ell}\circ \mathfrak{s}_{\ell}( \alpha_{s} \smile \langle \alpha_{i_1}, \dots, \alpha_{i_{\ell}} \rangle ) \in Z^3 ( F/F_\ell ;\Z) ,
\end{equation}
for any $(i_1 \cdots i_{\ell}) \in \mathfrak{U}_{\ell}, 1 \leq s \leq q $, and that the following set of 3-cocycles is a basis of the third cohomology $H^3( F/F_k ;\Z) \cong \oplus_{\ell =k}^{2k-2}\Z^{ q N_\ell -N_{\ell+1}} $.
\begin{equation}\label{bbbbb} \bigcup_{ k\leq \ell \leq 2k-2} \bigl\{ \ \mathfrak{s}_{\ell}\bigl( \alpha_{s}\smile \langle \alpha_{i_1}, \dots, \alpha_{i_{\ell }} \rangle \bigr) \ \bigr| \ \ (i_1 \cdots i_{\ell} ) \in \mathfrak{U}_{\ell}, \ \ 1 \leq s \leq q, \ \ (i_1 \cdots i_{\ell} s ) \not{\!\! \in} \ \! \mathfrak{U}_{\ell+1}\ \ \bigr\} .
\end{equation}
\end{enumerate}
\end{thm}
\begin{rem}The statements of (I) (II) are classically known; see \cite{FS,Orr2,Tur}.\end{rem}
\begin{proof} 
(I) \ Recall the definitions of (ii) and of \eqref{mag1231}. Thus, every lower Massey product is nullcohomologous by $a_{s,t}$ for some ($s,t$), by induction on $k$.

\noindent
(II) \ Denote by $\mathcal{S}$ the sum of $ c_{i_1 \cdots i_k}$ on $F_k/ F_{k+1} $, where $ i_1\cdots i_k \in \mathfrak{U}_k $. Namely, $ \mathcal{S}:= \oplus_{ I \in \mathfrak{U}_k} c_{I} : F_k/ F_{k+1} \ra \Z^{N_{k}} $. As is known \cite[Theorems 3.5 and 3.9]{CFL}, the sum $ \mathcal{S}$ is surjective; hence, it is bijective. Accordingly, the centrally extended group operation on $ F /F_k \times F_k /F_{k+1} $ from the 2-cocycles $ \oplus_{ I \in \mathfrak{U}_k} \langle \alpha_{i_1}, \dots, \alpha_{i_k} \rangle $ is, by definition, formulated as
\begin{equation}\label{bb4} ( g, \alpha ) \cdot ( h, \beta )= \bigl(g h , \mathcal{S}^{-1}\bigl( \bigoplus_{i_1\cdots i_k \in \mathfrak{U}_k} c_{i_1 i_{2}\cdots i_k} (\alpha + \beta) + \sum_{\ell: \ 1 \leq \ell < k} c_{i_1 i_{2}\cdots i_\ell} (g) c_{i_{\ell+1} \cdots i_k} (h ) \bigr) \bigr)
\end{equation}
for $g,h \in F/F_k$ and $\alpha , \beta \in F_k /F_{k+1}. $ Here, it is worth noticing that the subsequences $i_1 i_{2}\dots i_\ell$, $i_{\ell+1} \cdots i_k $ are also standard. Therefore, via the Magnus expansion, this group is isomorphic to the nature $F/ F_{k+1} $ as central extensions over $F/F_k$ (cf. matrix multiplications). Hence, the second cohomology $H^2(F/F_k ) \cong \Z^{N_k}$ is generated by the sum $ \oplus_{ I \in \mathfrak{U}_k} \langle \alpha_{i_1}, \dots, \alpha_{i_k} \rangle $, which provides a basis of $H^2(F/F_k ) $, as desired.

\noindent
(III) First, we will mention some of the results from \cite{IO}. Igusa and Orr \cite[Theorem 6.7]{IO} constructed a certain filtration on the homology, $ \mathcal{F}^{s} H_3( F/F_k ;\Z) $, such that two isomorphisms,
\begin{equation}\label{mag112} \frac{\mathcal{F}^{\ell}H_3 ( F/F_k ;\Z)}{ \mathcal{F}^{\ell+1}H_3 ( F/F_k ;\Z)} \cong \Z^{ q N_{\ell -1} -N_{\ell}} , \end{equation}
\[\mathcal{F}^{\ell}H_3 ( F/F_k ;\Z) \cong \mathrm{Im} ( (p_{\ell -1})_* :H_3 (F/F_{\ell-1} ;\Z ) \ra H_3 (F/F_k ;\Z ) ), \]
hold for $ k < \ell < 2k$, and the left quotients are zero otherwise. In particular, $\mathcal{F}^{\ell}H_3 ( F/F_k;\Z ) $ is a direct summand of $H_3( F/F_k ;\Z ) $, and $\Ker((p_{\ell -1})_* )$ is a direct summand of $H_3( F/F_{\ell -1} ) $, if $ k < \ell < 2k$. Thus, for $\ell \geq k$, we readily have the composite,
$$ H_3 ( F/F_k;\Z ) ) \xrightarrow{\mathrm{projection}} \mathcal{F}^{\ell +1}H_3 ( F/F_k ) \cong H_3 (F/F_{\ell} )/\Ker((p_{\ell})_*) \hookrightarrow H_3 ( F/F_{\ell} ;\Z ).$$
Applying of this composite to $\mathrm{Hom}(\bullet ,\Z)$ is regarded as a map $\iota_{\ell}: H^3( F/F_\ell ;\Z ) \ra H^3( F/F_k ;\Z)$. Accordingly, let us set up the composite of the cup product on $F/F_{\ell} $ and $\iota_{\ell}$:
$$\Theta_\ell : H^1 ( F/F_{\ell} ;\Z) \otimes H^2 ( F/F_{\ell} ;\Z) \stackrel{\smile}{\lra} H^3 ( F/F_{\ell} ;\Z) \stackrel{\iota_{\ell}}{\lra} H^3 ( F/F_k ;\Z). $$

We will show that the direct sum $\oplus_{\ell=k}^{2k-2} \Theta_\ell $ surjects onto $H^3 ( F/F_k ;\Z) $. Consider the Lyndon-Hochschild spectral sequence from the central extension $F_k/ F_{k+1} \ra F/F_{k+1} \ra F/F_k $. Then, as is shown in the proof of \cite[Lemma 5.8]{IO}, we find a sequence,
\begin{equation}\label{mag11} H_3(F/F_{k+1};\Z ) \lra H_3(F/F_k ;\Z) =E^2_{3,0} \stackrel{d^2_{3,0}}{\lra} E^2_{1,1} \ \stackrel{\delta_*}{\lra} H_2(F/F_k ;\Z ) \ra 0 \ \ \ \ \ \ \ \ (\mathrm{exact}),
\end{equation}
satisfying $ d^2_{3,0} (\mathcal{F}^{k+1}H_3 ( F/F_k ;\Z ) ) = \Ker(\delta_*)$ and $ d^2_{3,0} (\mathcal{F}^{k}H_3 ( F/F_k;\Z ) ) =0$. Keep in mind that each term is free. Furthermore, consider the spectral sequence on the cohomology level. Noting the identity
$$ E_2^{1,1} = H^1 ( F/F_k ; F_{k} /F_{k+1} ) \cong H^1 ( F/F_k;\Z ) \otimes H^2 ( F/F_k;\Z ) ,$$
we dually obtain the following sequence from \eqref{mag11}:
$$ 0 \lra H^2(F/F_k ) \lra H^1 ( F/F_k ) \otimes H^2 ( F/F_{k-1} ) \stackrel{d_2^{3,0}}{\lra} \mathcal{F}^{k+1} H^3(F/F_k ) \ \ \ \ \ \ \ \ \ (\mathrm{exact}). $$
As is usual with the $\mathrm{cup}_1$-product, the differential map $d_2^{3,0}$ is equal to the cup product. To summarize, this sequence and the isomorphisms \eqref{mag112} imply the surjectivity of $ \oplus_{\ell=k}^{2k-2} \Theta_{\ell} $, as required. Moreover, the construction of $\Theta_{\ell} $ admits the desired section $\mathfrak{s}_\ell$ satisfying \eqref{bbbbbb}.

The proof is completed by showing the linear independence of \eqref{bbbbb}, as follows. Notice from the definition of $ (i_1 \cdots i_{\ell} ) \in \mathfrak{U}_{\ell}$ that $(i_1 \cdots i_\ell s ) \in \mathfrak{U}_{\ell+1}$ if and only if $s > i_1$. Thus, if $s> i_1 $, then $i_\ell i_1 \cdots i_{\ell-1} s$ is not standard because of $i_k \geq s$. Here, we mention \cite[Proposition 4.5]{GL} showing the equality,
\[ \alpha_{s} \! \smile \! \langle \alpha_{i_1}, \dots, \alpha_{i_{\ell}} \rangle
= \alpha_{i_{\ell}} \smile \langle \alpha_{i_1}, \dots,\alpha_{i_{\ell-1}} , \alpha_{s}\rangle \in H^3(F/F_\ell ;\Z) . \]
Hence, $\alpha_{s} \! \smile \! \langle \alpha_{i_1}, \dots, \alpha_{i_{\ell}} \rangle$ with $ (i_1 \cdots i_{\ell} ) \in \mathfrak{U}_{\ell}$ and $(i_1 \cdots i_\ell s ) \in \mathfrak{U}_{\ell+1}$ is cohomologous to a cocycle in \eqref{bbbbb}. By \eqref{bbbbbb} and functoriality, a similar conclusion can be made even in the case $k < \ell \leq 2k-2$. Hence, from the surjectivity of $ \oplus_{\ell=k}^{2k-2} \Theta_{\ell} $, comparing the ranks of $H^3(F/F_k)$ with the order of \eqref{bbbbb} leads to linear independence, as required.
\end{proof}
We immediately obtain a corollary from the above proof.

\begin{cor}\label{stst} For any $s \leq q$,
the map $H^2 (F/F_k ;\Z ) \ra H^3 (F/F_k ;\Z ) $ which sends $\beta$ to $\alpha_{s} \smile \beta $ is injective.
\end{cor}
\begin{proof}The proof is straightforward from the basis in \eqref{bbbbb}. \end{proof}

\section{Applications: simple proofs}
\label{Sec3}
There are some topological invariants using $H^*(F/F_k)$, and some results on the Milnor link-invariants and the mapping class groups of surfaces. In this section, we give simple proofs of the results of Fenn-Sjerve \cite{FS}, Turaev \cite{Tur}, Porter \cite{Por}, Kitano \cite{Ki}, and Heap \cite{Heap}.

\subsection{Theorem of Fenn-Sjerve on the Massey product}
\label{Sec3222}
We state the theorem \cite{FS} and give an alternative proof using Theorem \ref{maainthm}. 
Assume $k \geq 3$. Let $W_1,\dots, W_t $ be words in $F_k$, and $R \subset F$ be the normal closure of $W_1,\dots, W_t $ and $F_{k+1}$. Let $G$ be the quotient group $F/R$. Since $H_1(F)\cong H_1(G) $, the 1-cocycle $\alpha_j : F\ra \Z $ induces the 1-cocycle $\alpha_j : G \ra \Z$ for $j \leq q. $

We will 
state Theorem \ref{ma33t22} below.
Let $p: G \ra F/F_k $ be the projection.
Recall Hopf's theorem which claims the isomorphisms,
$$H_2(G) \cong (R \cap [F,F])/[F,R], \ \ \ \ \ \ \ H_2(F/F_k) \cong (F_k \cap F_2)/[F,F_k] = F_k/F_{k+1} . $$
Noticing $ W_j \in R \cap [F,F]$, we denote the pushforward $p_*(W_j)$ by $\mathcal{W}_j$, and regard it as a 2-cycle of $F/F_k$.

\begin{thm}[\cite{FS}]\label{ma33t22}
Suppose that all of $W_1,\dots, W_t $ lie in $F_k$. For any $\ell < k $, every $\ell$-fold Massey product $ \langle \alpha_{i_1},\dots, \alpha_{i_{\ell}}\rangle$ vanishes. On the other hand, Massey products of length $k$ are defined and evaluated on $\{W_j\}$ according to the formula,
\begin{equation}\label{ff36}
\sum_{j_1,\dots, j_k} [ \langle \alpha_{j_1},\dots, \alpha_{j_k} \rangle,W_j 
\ ] X_{j_1} \cdots X_{j_k} = \mathcal{M} (\mathcal{W}_j) \in \mathcal{M} (F_k /F_{k+1}).
\end{equation}
Here, the outer $[ ,]$ is the pairing of $ H^2(G;\Z)$ and $ H_2(G;\Z)$.
\end{thm}
To prove this theorem, we give a lemma:
\begin{lem}\label{m3t22}
Take a standard index $I=i_1\cdots i_k \in \mathfrak{U}_k $,
and the associated Massey product $\langle \alpha_{i_1},\dots, \alpha_{i_k} \rangle$.
Via the isomorphism $H_2(F/F_k) \cong F_k/F_{k+1}$ above,
the Kronecker product $ [\langle \alpha_{j_1},\dots, \alpha_{j_k} \rangle, \bullet ]: H_2(F/F_k) \ra \Z$ coincides with the map $c_{i_1\cdots i_k } : F_k/F_{k+1} \ra \Z.$
\end{lem}
\begin{proof} For any $ a \in F/F_k$, we choose a representative $\bar{a} \in F$.
It follows from \cite[Exercise 4 in II.5]{Bro} that the correspondence $(g,h) \mapsto \overline{g} \overline{h} (\overline{gh})^{-1}$ induces the isomorphism $ H_2(F/F_k) \ra F_k/F_{k+1}$. By the cocycle expression of $\langle \alpha_{j_1},\dots, \alpha_{j_k} \rangle$ in Lemma \ref{maainlem}, a similar discussion to \eqref{bb4} readily deduces the required coincidence.
\end{proof}
\begin{proof}[Proof of Theorem \ref{ma33t22}]
Given a standard index $I=i_1\cdots i_k \in \mathfrak{U}_k $, let $W_I \in F_k/ F_{k+1}$ be a generator of the $i_1\cdots i_k $-th summand of $\Z^{N_k} \cong F_k/ F_{k+1}$. A previous paper \cite[\S 2]{CFL} explicitly describes the word $W_I $ and refers to it as {\it the standard commutator}.
The previous paper \cite[Lemma 3.4]{CFL} showed that, for any standard index $j_1\cdots j_k$, the $X_{j_1} \cdots X_{j_k}$-coefficient of $\mathcal{M} (W_I )$ is $ \delta_{i_1,j_1} \cdots \delta_{i_k,j_k} \in \{ 0,1\}$. By Lemma \ref{m3t22}, this is equal to $[ \langle \alpha_{j_1},\dots, \alpha_{j_k} \rangle,\mathcal{W}_I] $. In summary, if $G$ is presented by $ F/ \langle F_{k+1} , W_I \rangle $, the equality \eqref{ff36} holds.

Finally, we complete the proof for $G=F/R$. Since $ F_k/ F_{k+1}$ is abelian, we can expand $p_*(W_j)$ as $\prod_{I \in \mathfrak{U}_k} (W_I)^{a_I} $ for some $a_I \in \Z$. Then, from the discussion in the above paragragh, the equality \eqref{ff36} holds in the $X_{j_1} \cdots X_{j_k}$-coefficient with respect to every standard index $j_1\cdots j_k$. Since any other coefficient is a linear sum of the such $X_{j_1} \cdots X_{j_k}$-coefficients, we conclude that the equality \eqref{ff36} in $\mathcal{M} (F_k /F_{k+1}) $ is satisfied.
\end{proof}
\begin{rem}\label{nosa}
Finally, we should give a remark for the following subsections. Let us choose a connected CW complex $X$ such that $ \pi_1(X) \cong G$, and fix a classifying map $c: X \ra BG$ which is obtained by killing the higher homotopy groups of $X$.
Since $c_*: H_2(X)\ra H_2(BG;\Z)= H_2(G)$ is surjective, there are 2-cycles $\mathcal{X}_j \in H_2(X)$ with $ c_*( \mathcal{X}_j )= [W_j ]$. By $ \pi_1(X) \cong G$, the 1-cocycle $\alpha_{j}$ may be that of $H^1(X;\Z) $, and the pairing $[\langle \alpha_{j_1},\dots, \alpha_{j_k} \rangle,\mathcal{X}_j\ ]$ is equal to the equality \eqref{ff36}. To conclude, we can deal with some Massey products in $H^2(X;\Z)$ from the viewpoits of Theorem \ref{ma33t22}.
\end{rem}

\subsection{Milnor invariant and Massey product}
\label{Sec3222}
Porter \cite{Por} and Turaev \cite{Tur} independently showed that the Milnor link invariant is equivalent to some Massey products of the link complement space. For simplicity, this paper focuses on only the $k$-th leading terms of the Milnor invariant and gives an alternative proof of their result.

To state the theorems, we begin by reviewing the Milnor invariant, according to \cite{Mil2,Tur}. We suppose that the reader has elementary knowledge of knot theory, as provided in \cite{Por,Tur}. Let $M $ be an integral homology 3-sphere, that is, a closed 3-manifold such that $H_*(M;\Z) \cong H_*(S^3 ;\Z)$. Choose a link $L \subset M$ with $q$ components, and a meridian-longitude pair $(\mathfrak{m}_{\ell},\mathfrak{l}_{\ell} )$ for $ \ell \leq q $, where $\mathfrak{l}_{\ell}$ may be the preferred longitude. Then, it is known (\cite{Mil2}; see also \cite[Lemma 1.2]{Tur}) that the $m$-th nilpotent quotient $ \pi_1 (M \setminus L) / \pi_1 (M \setminus L)_m$ has the group presentation,
\begin{equation}\label{tautau22}
\bigl\langle \ \ x_1, \dots, x_q\ \ \bigl| \ \ [x_\ell, w_\ell^{(m)}]=1 \mathrm{ \ for \ } \ell \leq q , \ \ F_m \ \ \bigr\rangle,
\end{equation}
where $x_i$ is represented by the $j$-th meridian, and $ w_j^{(m)}$ is defined by the longitude in $\pi_1 (M \setminus L)$.

For brevity, let us assume the existence of $k \in \Z$ such that $ w_j^{(m)}$ is trivial in the $k$-th quotient for any $m \leq k$, i.e., $w_j^{(k)} \in F_k$. We call the existence {\it Assumption $\mathcal{A}_k$}. By considering $w_\ell^{(k)}$ to be a word in $F_k/F_{k+1} $, we will focus on the value $\mathcal{M} (w_\ell^{(k)}) \in \Z \langle X_1,\dots, X_q \rangle /\mathcal{J}_{k+1 } $. The coefficient of $X_{i_1} \cdots X_{i_k}$ of $\mathcal{M}(w_\ell^{(k)}) $ is called {\it the $k $-th Milnor $\mu $-invariant} of $L $, and it is denoted by $\mu(i_1 \cdots i_k;\ell )$. Let $\alpha_{j} : \pi_1(M \setminus L;\Z ) \ra \Z$ be the
homomorphism which sends $ \mathfrak{m}_{\ell}$ to $\delta_{\ell ,j}$.

\begin{thm}[{The minimal non-vanishing case of \cite{Tur}\cite{Por}}]\label{maaint22} Suppose $\mathcal{A}_k$. Let $[\mathfrak{l}_{\ell}] \in H_2(M \setminus L ;\Z) $ be a 2-cycle corresponding to the $\ell$-th longitude $\mathfrak{l}_{\ell} $.
For any index $I= i_1 i_2 \cdots i_k$, the $(k+1)$-fold Massey product $\langle \alpha_{i_1},\dots,\alpha_{i_{k-1}}, \alpha_{i_k},\alpha_{\ell} \rangle$ is uniquely defined, and the following equality holds:
\begin{equation}\label{taue22}
\mu(i_1 \cdots i_k;\ell ) = \ [ \langle \alpha_{i_1},\dots,\alpha_{i_{k-1}}, \alpha_{i_k},\alpha_{\ell} \rangle, [\mathfrak{l}_{\ell}] ] \in \Z.
\end{equation}
\end{thm} 
\begin{proof}
Denote $\pi_1 (M \setminus L)$ by $G$, and take the classifying map $\mathcal{C}: M\setminus L \ra BG $, explained in Remark \ref{nosa}. We claim that the pushforward $ \mathcal{C}_*([\mathfrak{l}_{\ell}] )$ corresponds to the relation $[\mathfrak{m}_\ell, \mathfrak{l}_\ell] $. Consider the $\ell$-th torus boundary, $\partial_{\ell} (M\setminus L)$, as a $K(\Z^2,1)$-space. This $\pi_1( \partial_{\ell} (M\setminus L))\cong \Z^2$ is presented by $\langle \mathfrak{m}_{\ell},\mathfrak{l}_{\ell} \ | \ \mathfrak{m}_{\ell}\mathfrak{l}_{\ell}\mathfrak{m}_{\ell}^{-1}\mathfrak{l}_{\ell}^{-1} \rangle $. Following Hopf's theorem, the generator of $H_2( \partial_{\ell} (M\setminus L)) \cong \Z$ corresponds to $\mathfrak{m}_{\ell}\mathfrak{l}_{\ell}\mathfrak{m}_{\ell}^{-1}\mathfrak{l}_{\ell}^{-1} $. Considering the pushforwards via the inclusion $ \partial_{\ell} (M\setminus L) \ra M\setminus L$, we will prove the claim using $\mathcal{C}$.
Let $p: G \ra G/G_{k+1}$ be the projection. By the above claim, the pushforward $(p\circ \mathcal{C})_*( [\mathfrak{l}_{\ell}] )$ corresponds to the word $[x_\ell, w_\ell^{(k)}] $.

We are now in a position to complete the proof.
By the assumption, $w_{\ell}^{(k)}$ lies in $ F_{k}$. Therefore, $ [x_\ell, w_{\ell}^{(k)}] \in F_{k+1 }$. 
Hence, it follows from the theorem \ref{ma33t22} of Fenn-Sjerve that
$\langle \alpha_{i_1},\dots, \alpha_{i_k},\alpha_{\ell} \rangle$ of length $k+1$ is uniquely defined; see Remark \ref{pp2}. Moreover, the left hand side of \eqref{taue22} equals the right hand side of \eqref{ff36},
and the right one means the $X_{ i_1} \cdots X_{ i_k} X_{\ell}$-coefficient of $\mathcal{M}([x_\ell, w_{\ell}^{(k+1)}])$. Then, we can verify, by directly computing $\mathcal{M}(x_\ell w_\ell^{(k+1)}x_\ell^{-1} (w_\ell^{(k+1)})^{-1})$, that it is equal to the $X_{ i_1} \cdots X_{ i_k} $-coefficient of $\mathcal{M}(w_\ell^{(k)})$, that is, $\mu(i_1 \cdots i_k;\ell ) $, as required.
\end{proof}

Incidentally, Milnor defined the $\mu$-invariant of higher degree, and Porter \cite{Por} and Turaev \cite{Tur} described the higher invariant in terms of Massey product; see also the paper \cite{KN}, which provides a refinement of the higher invariant. Although we omit showing the details, the higher degree relation can be proven in the same manner as Theorem \ref{maaint22}.

\subsection{Johnson homomorphisms and Massey products}
\label{Se80222}
Now let us focus on the Johnson homomorphisms of the mapping class groups of surfaces; see, e.g., \cite{Joh,Morita1,Day,GL,Heap,Morita2} for the significance of these homomorphisms. This section paraphrases \cite{Ki}, which describes the relation between the Johnson homomorphisms and Massey products.

First, we should establish the notation. Suppose $q = 2g$. Let $\Sigma_{g,r}$ be a compact oriented surface of genus $g$ with $r$ boundaries. Let $\Gamma_{g,1}$ be the group of isotopy classes of orientation-preserving homeomorphisms of $\Sigma_{g,1}$. Then, the action of $\Gamma_{g,1}$ on $F=\pi_1(\Sigma_{g,1} ) $ can be regarded as a homomorphism $\Gamma_{g,1}\ra \mathrm{Aut}(F)$. Subject to $F_k$, we have $\rho_k: \Gamma_{g,1}\ra \mathrm{Aut}(F/F_k).$ We commonly denote the kernel $\Ker(\rho_k)$ by $\mathcal{T}(k)$. We have the filtration $\Gamma_{g,1} \supset \mathcal{T}(2)\supset \mathcal{T}(3)\supset \cdots.$

Next, let us review the homomorphism \eqref{tautau} below. Fix any $f \in \mathcal{T}(k) $. Given $[\gamma] \in H_1(\Sigma_{g,1} )=F/F_2$, choose a representative $\gamma \in \pi_1(\Sigma_{g,1})$. Then, $f_* (\gamma)\gamma^{-1}$ lies in $F_k$ since the action of $f_* $ by $f \in \mathcal{T}(k)$ on $F/F_k$ is trivial. Then, {\it the $k$-th Johnson homomorphism} is defined as the map,
\begin{equation}\label{tautau}
\tau_k: \mathcal{T}(k)\lra \mathrm{Hom} (H_1(\Sigma_{g,1} ), \ \ F_k/F_{k+1}),
\end{equation}
which sends $[f]$ to the homomorphism $[\gamma] \ra [f_*(\gamma) \gamma^{-1}] $. As is known, this $\tau_k $ is well-defined and a homomorphism. The following is also well-known; see \cite{Joh2,Morita2}.
For $ f \in \mathcal{T}(k) $, this$ f $ lies in $\mathcal{T}(k+1 ) $ if and only if $ \tau_m (f)=0$ for any $m \leq k$.

Next, let us examine the mapping torus $T_{f,1}$ for a fixed $f \in \mathcal{T}(k) $ with $k \geq 2$. Here, $T_{f,1}$ is the quotient space of $\Sigma_{g,1} \times [0,1]$ subject to the relation $(y,0) \sim (f(y),1)$ for any $y \in \Sigma_{g,1}$. Since $f \in \mathcal{T}(k) $ with $k \geq 2$, we have $ H_*(T_{f,1}) \cong H_*(\Sigma_{g,1} \times S^1). $ Furthermore, fix a basis $\{ x_1, \dots, x_{2g} \}$ of the free group $F= \pi_1(\Sigma_{g,1}) .$ Then, following the van Kampen argument, one can verify the presentation,
\begin{equation}\label{m3as2} \pi_1(T_{f,1}) \cong \langle x_1, \dots, x_{2g} , \gamma \ \ | \ \ [x_1, \gamma] f_*(x_1) x_1^{-1}, \dots,[x_{2g}, \gamma] f_*(x_{2g}) x_{2g}^{-1}
\ \rangle.
\end{equation}
Here, $\gamma$ represents a generator of $\pi_1(S^1). $ Since $T_{f,1}$ is a $\Sigma_{g,1}$-bundle over $S^1$ by definition, it is a $K(\pi,1)$-space. Hence, $H_*(\pi_1(T_{f,1})) \cong H_*(T_{f,1}) \cong H_*(\Sigma_{g,1} \times S^1). $ Moreover, Hopf's theorem implies that the relations $ [x_1, \gamma] f_*(x_1) x_1^{-1}, \dots,[x_{2g}, \gamma] f_*(x_{2g}) x_{2g}^{-1} $ represent a basis $\{\mathcal{X}_1,\dots, \mathcal{X}_{2g} \}$ of $ H_2 (T_{f,1}) \cong \Z^{2g}$.

Now let us state and prove Proposition \ref{maai322}. Since the boundary $\partial T_{f,1}$ is the torus $ \partial \Sigma_{g,1} \times S^1$, we can define $T^{\gamma}_{f}$ to be the resulting space obtained by filling in the torus $\partial T_{f,1} \simeq \partial \Sigma_{g,1} \times S^1 $ with the solid torus $\partial \Sigma_{g,1} \times D^2 $. This $T^{\gamma}_{f}$ is called {\it the Dehn filling along a curve on} $\partial T_{f,1}$. By denoting the inclusion $T_{f,1} \ra T^{\gamma}_{f}$ by $ \iota^{\gamma}$, the homology $H_2(T^{\gamma}_{f};\Z ) \cong H^1(T^{\gamma}_{f};\Z ) \cong \Z^{2g}$ is spanned by the pushforwards $\{ \iota^{\gamma}_*(\mathcal{X}_1),\dots, \iota^{\gamma}_*(\mathcal{X}_{2g}) \}$. Moreover, we should notice the presentation of $ \pi_1(T^{\gamma}_{f})$,
\begin{equation}\label{m3a522} \pi_1(T^{\gamma}_{f}) \cong \langle \ x_1, \dots, x_{2g} \ \ | \ \ f_*(x_1) x_1^{-1}, \dots, f_*(x_{2g}) x_{2g}^{-1}
\ \rangle.
\end{equation}

\begin{prop}
[{cf. \cite{Ki} and \cite[Corollary 4.1]{GL}}]\label{maai322} Let $ \{\mathcal{X}_1,\dots, \mathcal{X}_{2g} \}$ be the basis of $ H_2 (T_{f,1}) $, and $\iota^{\gamma} : T_{f,1} \ra T^{\gamma}_{f}$ be the inclusion as above. Let $ x_j^* : \pi_1(T^{\gamma}_{f}) \ra \Z$ be the 1-cocycle which sends $x_i$ to $\delta_{j,i}$. Define the map $\tau_k' :\mathcal{T}(k) \ra \mathrm{Hom}(H_1(\Sigma_{g,1} ), \ \mathcal{M}(F_k/F_{k+1}))$ by letting $\tau_k'(f) $ be the homomorphism,
\begin{equation}\label{ffff} x_i \longmapsto \sum_{j_1,\dots, j_k} [ \langle x_{j_1}^*, \dots, x_{j_k}^* \rangle, \iota^{\gamma}_*(\mathcal{X}_i)
\ ] X_{j_1} \cdots X_{j_k}.
\end{equation}
Here, the outer $[,] $ is the pairing of $ H^2(T_{f}^{\gamma})$ and $ H_2(T_{f}^{\gamma})$. Then, $\tau'_k (f) (x_i)$ is equal to $\mathcal{M}\bigl( \tau_k (f) (x_i) \bigr)$.
\end{prop}

\begin{proof}
As mentioned above, $ f_*(x_i) x_i^{-1} \in F_k$, since $ f\in \mathcal{T} (k)$. The statement readily follows from Theorem \ref{ma33t22} and Remark \ref{nosa} with $W_j=f_*(x_j) x_j^{-1} $ and $t =2g$.
\end{proof} 

\subsection{Vanishing condition of the Morita homomorphism}
\label{Se9312}
As a lift of the Johnson homomorphism $ \tau_k $, Morita \cite{Morita2} defined a map $ \tilde{ \tau}_k: \mathcal{T}(k) \ra H_3(F/F_k), $ which is called {\it the Morita homomorphism}. Furthermore, Heap \cite{Heap} showed the vanishing condition of $ \tilde{ \tau}_k$ in terms of (relative) bordism theory. The purpose of this subsection is to give a simpler proof of the result (Theorem \ref{maaiheap}).

Let us review the map $ \tilde{ \tau}_k $. Fix $f \in \mathcal{T}(k) $. 
Thus, the relation $ f_*(x_j) x_j^{-1}$ vanishes in the $k$-th nilpotent quotient of $\pi_1(T^{\gamma}_{f}) $. Thus, from \eqref{m3a522}, we have the canonical surjection,
$$ \phi^{\gamma}_{f,k} : \pi_1(T^{\gamma}_{f}) \lra \pi_1(T^{\gamma}_{f})/ \bigl( \pi_1(T^{\gamma}_{f})\bigr)_k \cong F/F_k.$$
Let $[ T^{\gamma}_{f}] \in H_3(T^{\gamma}_{f}) \cong \Z$ be the fundamental 3-class. Then, $ \tilde{ \tau}_k (f) $ is defined to be the pushforward $(\phi^{\gamma}_{f,k} )_* ([T^{\gamma}_{f}]) \in H_3(F/F_k)$. It is known \cite{Morita2,Heap} that this $ \tilde{ \tau}_k (f) $ depends only on $f \in \mathcal{T}(k) $ and that $ \tilde{ \tau}_k $ is a lift of the Johnson map $ \tau_k $.

\begin{thm}
[{\cite[Theorem 5]{Heap}}]\label{maaiheap} Let $f \in \mathcal{T}(k)$. Then, $ \tilde{ \tau}_k (f) = 0$ if and only if $f \in \mathcal{T}(2k-1 )$.
\end{thm}

\begin{proof}
First, let us make some observations. Notice from this Proposition \ref{maai322} that $f \in \mathcal{T}(k+1 )$ if and only if the pairing $ [ \langle x_{j_1}^*, \dots, x_{j_{k}}^* \rangle, \iota_*^{\gamma}(\mathcal{X}_i) \ ] $ in \eqref{ffff} is zero for any standard indexes $j_1 \cdots j_k $. Furthermore, Theorem \ref{maainthm} (III) implies that $ \tilde{ \tau}_k (f) $ is zero if and only if the pairing
\begin{equation}\label{ff} \langle \mathfrak{s}_\ell ( x_s^* \smile \langle x_{j_1}^*, \dots, x_{j_\ell}^* \rangle ), \ \ (\phi^{\gamma}_{f,k} )_* ([T^{\gamma}_{f}]) \ \rangle \end{equation}
is zero for any sequences $(s, j_1, \cdots, j_{\ell}) $ in \eqref{bbbbb}, where we replace $\alpha_j$ by $x_j^* $.

We now complete the proof. Suppose $f \in \mathcal{T}(2k-1 )$. Then, we can define the Johnson homomorphism $\tau_\ell$ for $\ell \leq 2k-2$ and have $\tau_\ell=0 $. Hence,  all the pairing $ [ \langle x_{j_1}^*, \dots, x_{j_\ell}^* \rangle, \iota_*^{\gamma}( \mathcal{X}_s ) \ ] $ is zero. Note the cap product $ [T^{\gamma}_{f}] \cap \alpha_s = \iota_*^{\gamma}( \mathcal{X}_s ) $ because $H^*(T^{\gamma}_{f};\Z ) \cong H^*(S^1 \times \Sigma_{g,0};\Z ) $. Thus, the pairing in \eqref{ff} is zero. Hence, $ \tilde{ \tau}_k (f)=0 $. Conversely, let us assume $ \tilde{ \tau}_k (f)=0 $. Then, all of the pairing \eqref{ff} is zero. When $\ell =k$, Corollary \ref{stst} implies that the pairing \eqref{ffff} is zero. From Theorem \ref{maai322}, we have $\tau_k(f)=0$, i.e., $f \in \mathcal{T}(k+1)$; thus, we can define $\tau_{k+1}(f)$. In a similar way, we can define $\tau_{\ell}(f)$ and show $\tau_\ell(f)=0 $ for any $\ell \leq 2k-2$. Hence, we conclude $ f \in \mathcal{T}(2k-1 ) $, as desired.
\end{proof}

Finally, we mention some results on the Morita maps $\tilde{\tau}_k$.
Some papers \cite{Morita2,Heap,Day} studied the maps. 
Especially, Massuyeau \cite{Mas} showed an equivalence between
the map $\tilde{\tau}_k$ and some degrees of ``the total Johnson homomorphism", and
give a computation of $\tilde{\tau}_k$ with rational coefficients, in terms of a symplectic expansion.
However, our situation is with integral coefficients; if we can determine the homeomorphism type of $T_{f}^{\gamma}$, we can hope to compute $\tilde{\tau}_k(f)$ by using the 3-cocycles of Theorem \ref{maainthm} (III).

\section{Expressions of some 3-cocycles}
\label{So}
In general, it is practically important to give explicit expressions of group cocycles. This section focuses on quotient groups of $F/F_k$ and gives an algorithm to describe their 3-cocycles. As mentioned in the introduction, we regard the higher Massey products as an algorithm to produce cocycles. 

\subsection{3-cocycles of $F/F_k$}
\label{Sbra2}
First, let us focus on the group $F/F_k$ and give presentations of the 3-cocycles $\mathfrak{s}_k ( \alpha_{s}\smile \langle \alpha_{i_1}, \dots, \alpha_{i_{k}} \rangle ) $ in Theorem \ref{maainthm}, when $\ell = k$ and $\ell = k+1$ with $k \geq 3$.

Using the notation in \S 3, given an index $ i_1i_2 \cdots i_\ell$ and $s \in \mathbb{N}$, let us define the map,
$$ \Gamma_{ s i_1i_2 \cdots i_\ell} : (F/F_\ell)^3 \lra \Z ; \ \ \ \ \ \ (x,y,z) \longmapsto c_{s}(x) \bigl( \sum_{j: \ 1 \leq j \leq \ell -1} c_{i_1 i_{2}\cdots i_j} (y)c_{i_{j+1} \dots i_\ell} (z)\bigr). $$
The simplest case is when $\ell=k$, in which the 3-cocycle $\mathfrak{s}_k ( \alpha_{s}\smile \langle \alpha_{i_1}, \dots, \alpha_{i_{\ell}} \rangle ) $ is exactly presented as $\Gamma_{ s i_1i_2 \cdots i_k}$, since $\mathfrak{s}_k =\mathrm{id}$. Next, to get presentations of $\mathfrak{s}_\ell ( \alpha_{s}\smile \langle \alpha_{i_1}, \dots, \alpha_{i_{\ell}} \rangle ) $ with $ k< \ell\leq 2k-2$, it is enough to explicitly give a function $\mathfrak{b}: (F/F_\ell )^2 \ra \Z $ such that the difference $\Gamma_{ s i_1i_2 \cdots i_\ell} - \partial^*_2(\mathfrak{b}) $ is a restricted map $(F/F_k)^3 \ra \Z $.

For example, we will describe the case $\ell =k+1$. Define the map $\mathfrak{b}$ by setting
$$ \mathfrak{b}(x,y)= c_{s} (x) c_{i_1 \cdots i_{k+1}} (y) c_{i_{k+2}} (y)+ c_{s i_1} (x) c_{i_2 \cdots i_{k+2}} (y) + c_{s i_{k+2}} (x) c_{i_1 \cdots i_{k+1}} (y) . $$
Then, as a result of the difference $( \Gamma_{ s i_1i_2 \cdots i_{k+1}}-\partial^*_2 \mathfrak{b} ) (x,y,z)$, we obtain

\begin{prop}\label{maainthm22}
The cohomology 3-class $\alpha_{s}\smile \langle \alpha_{i_1}, \dots, \alpha_{i_{k}} \rangle$ is represented by $\Gamma_{ s i_1i_2 \cdots i_k}$.

When $\ell =k+1$, the 3-class $\mathfrak{s}_{k+1} ( \alpha_{s}\smile \langle \alpha_{i_1}, \dots, \alpha_{i_{k+1}} \rangle )$ is represented by the map:
\[(x,y,z) \longmapsto \sum_{\ell=2}^{k} \Bigl( c_{s} (x) \bigl( c_{i_1 \cdots i_{\ell}} (y) c_{i_{\ell+1} \cdots i_{k+1}} (z) - c_{i_1 \cdots i_{\ell-1}} (y) c_{i_{\ell} \cdots i_{k+1}} (z) ( c_{i_{k+1}} (y) + c_{i_{k+1}} (z)) \bigr) \] \[ \ \ \ \ \ \ \ \ \ \ \ \ \ \ \ \ \ \ \ \ \ \ - c_{s i_1} (x) c_{i_2 \cdots i_{\ell}} (y) c_{i_{\ell+1} \cdots i_{k+1}} (z) - c_{s i_{k+1}} (x) c_{i_2 \cdots i_{\ell-1}} (y) c_{i_{\ell} \cdots i_{k}} (z) \Bigr) . \]
Since the length of every sequence in each term is less than $k$, this map $\Gamma_{s i_1i_2 \cdots i_{k+1}}-\partial^*_2 \mathfrak{b} $ can be regarded as a map from $ (F/F_k)^3 $.
\end{prop}
Here, we should mention that while Igusa and Orr \cite[\S 10]{IO} express the 3-cocycles $\alpha_{s}\smile \langle \alpha_{i_1}, \dots, \alpha_{i_{k}} \rangle $ in terms of Igusa's picture, our description using Massey products is simpler and compatible with the non-homogenous complex of $G$.

Concerning the higher case $\ell > k+1$, the author attempted to describe the 3-cocycles, but made little progress.

\subsection{Quotient groups by central elements}
\label{So1}
This subsection deals with the situation in \S \ref{Sec3222} or \cite{FS}. Namely, we fix central elements $W_1 ,\dots, W_t \in F_k/F_{k+1}$ and let $G$ be the quotient group of $F/F_{k+1} $ subject to $W_1 ,\dots, W_t $. For simplicity, let us assume $k>2$ and that there are standard sequences $ I^{(j)}= (i_1^{(j)},\dots, i_{k}^{(j)}) \in \mathfrak{U}_{k}$ with $j \leq s $, which are mutually distinct, and that $W_j$ corresponds to the Massey product $ \langle \alpha_{i_1^{(j)}} ,\dots, \alpha_{i_{k}^{(j)}} \rangle$ of length $k$. Such an assumption appears in discussions on higher Milnor invariant; see \cite{Tur,KN}.

Then, as in \eqref{mag1231}, we can easily check that the following map is well-defined and a 2-cocycle:
$$ \phi_{\Lambda_j}: G \times G \lra \Z ; \ \ \ ([X],[Y]) \longmapsto \sum_{\ell=1}^{k-1} c_{i_1^{(j)} i_{2}^{(j)}\cdots i_\ell^{(j)}} (X) c_{i_{\ell+1}^{(j)} \cdots i_{k}^{(j)}} (Y),
$$
where we represent any element of $G $ by the representative from $F/F_k.$ Here, we should mention the 5-term exact sequence from the central extension $ F/F_k\ra G$:
$$0 \ra H^1 (G;\Z) \stackrel{\cong}{\lra} H^1 (F/F_k ;\Z ) \lra \Z^{m} \stackrel{\delta^*}{\lra} H^2 (G;\Z ) \ra H^2 (F/F_k ;\Z ). $$
Actually, we can verify that the 2-cocycle $ \phi_{\Lambda_j}$ corresponds to the image of $\delta^*.$

Now let us give some 3-cocycles of $G$.

\begin{prop}\label{mad222}
Let $G$ be the group $F/\langle F_k, W_1 ,\dots, W_s \rangle $, as above. Fix $r,s\in \{ 1, \dots, q\}$ such that $(r i_1^{(j)}\dots i_{k-1}^{(j)}) $ and $(i_2^{(j)}\dots i_{k}^{(j)}s) $ are different from other indexes $i_1^{(j')} \cdots, i_{k}^{(j')}$ for $j'\leq s $. Then, the Massey product $ \langle \alpha_r, \phi_{\Lambda_j} , \alpha_s \rangle $ is defined and is represented by the map,
$$(x,y,z) \longmapsto c_r(x) \bigl( \sum_{\ell: \ 1 \leq \ell<k}\!\!\! c_{ i_1^{(j)} \cdots i_{\ell}^{(j)}} (y)c_{i_{\ell+1}^{(j)} \cdots i_k^{(j)} s} (z) \bigr) - \bigl( \sum_{\ell : \ 1 <\ell \leq k}\!\!\! c_{r i_1 ^{(j)} \cdots i_\ell^{(j)}} (x)c_{i_{\ell+1}^{(j)} \cdots i_k^{(j)}} (y) \bigr) c_s(z). $$
\end{prop}
\begin{proof}
Notice the equalities,
$$\alpha_r \smile \phi_{\Lambda_j} = \partial^* \bigl( \sum_{\ell: \ 1 \leq \ell<k}\!\!\! c_{r i_1^{(j)} \cdots i_\ell^{(j)}} (x)c_{i_{\ell+1}^{(j)} \cdots i_k^{(j)}} (y) \bigr), \ \ \ \ \ \ \ \
\phi_{\Lambda_j} \smile \alpha_s= \partial^* \bigl( \sum_{\ell: \ 1 < \ell \leq k}\!\!\! c_{ i_1^{(j)} \cdots i_\ell^{(j)}} (x)c_{i_{\ell+1}^{(j)} \cdots i_k^{(j)} s} (y) \bigr).
$$
Then, from the definition of the triple Massey product, we have a representative.
\end{proof}

\subsection*{Acknowledgments}
The author expresses his gratitude to Teruaki Kitano and
Gw\'{e}na\"{e}l Massuyeau for valuable comments on this paper.
He is also indebted to anonymous referees of an earlier version of this paper for providing insightful comments.

\appendix
\section{Cocycles in de Rham complexes of the iterated torus bundle}
\label{App}
We will describe the 2-cocycles in Theorem \ref{maainthm} as differential forms in de Rham complexes.

First, when $G=F/F_k$, we can verify, by induction on $k$, that the Eilenberg-MacLane space $B(F/F_k)$ can be realized as a $C^{\infty}$-manifold. Precisely, since $B (F_{k}/F_{k+1})$ is the product of $N_k$-copies of $S^1$, the geometric realization of \eqref{kihon2} implies that $B(F/F_{k+1} )$ is a universal torus bundle over $B(F/F_{k}) $.

We will address the manifold structure of $B(F/F_{k})$ in detail. The Heisenberg (triangular) group is a good toy model from the viewpoint of the unipotent Magnus expansion $\Upsilon_k $. Here, we should review generalized shuffles and the image of the Magnus expansion $\mathcal{M}$. As in \cite[\S 2]{CFL}, a sequence $(c_1c_2 \cdots c_k) \in \{ 1,\dots ,q\}^k$ is called {\it the resulting shuffle of two sequences} $I=a_1 a_2 \cdots a_{|I|} $ and $J=b_1b_2 \cdots b_{|J|}$ if there are $|I|$ indices $\alpha (1), \alpha(2), \dots ,\alpha(|I|)$ and $|J|$ indices $\beta (1), \beta (2), \dots, \beta (|J|)$ such that
\begin{enumerate}[(i)]
\item $ 1 \leq \alpha (1) < \alpha (2)< \cdots < \alpha (|I|) \leq k,$ \ and $ \
\ 1 \leq \beta (1) < \beta (2)< \cdots < \beta (|J|) \leq k$,
\item $c_{\alpha(i)}= a_i$ and $c_{\beta(j)}= b_j$ for all $i \in \{ 1,2, \dots, |I|\}$ \ $j \in \{ 1,2, \dots, |J|\} $,
\item each index $s \in \{ 1, 2, \dots ,k \}$ is either an $\alpha (i)$ for some $i$ or a $\beta(j)$ for some $j$ or both.
\end{enumerate}
Let the symbol $\mathrm{Sh} (I,J)$ denote the set of the resulting shuffles of $I$ and $J$. Thanks to \cite[Theorem 3.9]{CFL}, the image of the Magnus expansion $\mathcal{M} $ completely characterizes ``the generalized shuffle relation". In fact, the image in $\Z \langle X_1,\dots, X_q \rangle /\mathcal{J}_{k} $ is realized as
\begin{equation}\label{image} \Bigl\{ \sum_{ I= (i_1 \cdots i_n)} a_I \cdot X_{i_1} \cdots X_{i_n} \ \Bigl| \ \mathrm{For \ any \ indexes} \ J \mathrm{\ and \ } K, \ \ \ a_J \cdot a_K =\sum_{L \in \mathrm{Sh}(J,K)} a_L \ \Bigr\}.
\end{equation}

We further examine the $\R$-extension of the image and state Theorem \ref{32}. Let us consider the real extension $\R \langle X_1,\dots, X_q \rangle /\mathcal{J}_{k} = \R \otimes \Z \langle X_1,\dots, X_q \rangle /\mathcal{J}_{k} $ and identify it with the Euclid space of dimension $\sum_{\ell=0}^{k-1} q^\ell.$ 
Furthermore, we will define a subspace, $\mathrm{Im}(\mathcal{M}_{\R,k})$, of $\R \langle X_1,\dots, X_q \rangle /\mathcal{J}_{k} $ as follows.
Let $\mathrm{Im}(\mathcal{M}_{\R,k})$ with $k=2 $ be $\R \langle X_1,\dots, X_q \rangle /\mathcal{J}_{2} \cong \R^q$.  Supposing the definition of $\mathrm{Im}(\mathcal{M}_{\R,k-1}) $, we define $\mathrm{Im}(\mathcal{M}_{\R,k}) \subset \R \langle X_1,\dots, X_q \rangle /\mathcal{J}_{k} $ by 
\begin{equation}\label{image2} \notag  \Bigl\{ \sum_{ I= (i_1 \cdots i_n)} \!\!\!a_I \cdot X_{i_1} \cdots X_{i_n} \ \Bigl| 
\ \begin{array}{l}
\ a_{i_1 \cdots i_{n}} \in \mathcal{M}_{\R,k-1}, \ \mathrm{if \  } n <k.   \\
\mathrm{For \ any } \ J \mathrm{\ and \ } K \ \mathrm{ with }\ |J| +|K| = n, \ \ \ a_J \cdot a_K =\sum_{L \in \mathrm{Sh}(J,K)} a_L \
\end{array}  \Bigr\} .
\end{equation}
Then, as in the fact \cite[\S 3]{CFL} that the shuffle ration is closed under the multiplication of $(\Z[X_1,\dots, X_q \rangle /\mathcal{J}_{k})^{\times  }$, so is the closed set $\mathrm{Im}(\mathcal{M}_{\R,k}) $ 
under that of $(\R[X_1,\dots, X_q \rangle /\mathcal{J}_{k})^{\times  }$. Hence, this $\mathrm{Im}(\mathcal{M}_{\R,k}) $ is a Lie group, which contains $ \mathrm{Im }(\mathcal{M}_k)$ as a lattice. In other words, $ F/F_k $ acts freely and properly on $ \mathrm{Im }(\mathcal{M}_k)$. 
\begin{thm}\label{32}
The quotient space of $\mathrm{Im}(\mathcal{M}_{\R,k})$ subject to the free action of $ F/F_k $ is a closed connected $C^{\infty}$-manifold and is an Eilenberg-MacLane space of $ F/F_k$.
\end{thm}
\begin{proof}It is enough to show that $\mathrm{Im}(\mathcal{M}_{\R,k})$ is contractible, 
by induction on $k$. 
First, if $k=2$, $\mathrm{Im}(\mathcal{M}_{\R,k})$ is homeomorphic to $\R^q$ by definition. 
Next, consider the projection 
$p_k: \mathrm{Im}(\mathcal{M}_{\R,k}) \ra \mathrm{Im}(\mathcal{M}_{\R,k-1}) $.
For any $ x\in \mathrm{Im}(\mathcal{M}_{\R,k-1}) $, $p^{-1}_k(x)$ 
can be regarded as the subspace of solutions of a linear equation, by definition of the shuffle relation. Moreover, the rank of the subspace does not depend on $x$, by \eqref{kihon2}. 
Thus, $p_k$ is a $\R$-vector bundle. Hence, $\mathrm{Im}(\mathcal{M}_{\R,k})$ is contractible, by induction. 
\end{proof}

As a result of Theorem \ref{32}, every $s$-form of $BF/F_k$ is identified with an $ F/F_k$-invariant $s$-form of $\mathrm{Im}(\mathcal{M}_{\R,k})$. Thus, via the de Rham theorem, we will describe the basis of $H^*( F/F_k;\R ) \cong H^*_{\rm dR}( BF/F_k) $ as $ F/F_k$-invariant $s$-forms of $\mathrm{Im}(\mathcal{M}_{\R,k})$ as follows:

To state Lemmas \ref{maas22} and \ref{ma3188}, we need some terminology. Consider the cotangent bundle of $\R \langle X_1,\dots, X_q \rangle /\mathcal{J}_{k} $, and denote by $d X_{j_1 \dots j_t}$ the dual basis corresponding to the coordinate $ X_{j_1} X_{j_2} \cdots X_{j_t} $. Following the pullback, we regard the basis as 1-forms in $T^* \mathrm{Im}(\mathcal{M}_{\R,k})$. Furthermore, for $s \in \{ 1, \dots, q\}$, the 1-form $d X_{s} $ on $\mathrm{Im}(\mathcal{M}_{\R,k})$ is $ F/F_k$-invariant, and the resulting 1-cocycle in $\bigwedge^1B F/F_k$ corresponds to the $s$-th summand of the abelianization $\alpha_s$. Extend the map $\beta_{i_u i_{u+1} \dots i_v}$ in \eqref{mag1231} as $ \R \langle X_1,\dots, X_q \rangle /\mathcal{J}_{k} \ra \R$. In addition, for $(t,k_0) \in \mathbb{N}^2 $, we prepare a set of the form,
$$ \mathcal{S}_{t,k_0}:= \{ \ (k_1, \dots,k_u ) \in \mathbb{N}^u \ | \ u \geq 1, \ \ \ k_0+ k_1 + \cdots + k_u =t \ \}. $$

\begin{lem}\label{maas22}
Fix an index $(j_1 , \dots, j_t) \in \{ 1, \dots, q\}^{t}$. Define the 1-form of the formula,
$$ \sum_{k_0=1}^t ( \sum_{(k_1, \dots, k_u) \in \mathcal{S}_{t,k_0},}\!\! (-1)^{u} \!\! \prod_{w : \ 1 \leq w \leq u} \beta_{j_{1+k_0+k_1 + \cdots + k_{w-1}} j_{2+k_0+k_1 + \cdots + k_{w-1}} \cdots j_{k_0+k_1 + \cdots + k_{w}}} d X_{j_1 \cdots j_{k_0}} ) .$$
We denote this 1-form by $\gamma_{j_1 \cdots j_t}$. This 1-form is $F/F_k$-invariant.
\end{lem}
In what follows, we denote the sum $k_0 + k_1 + \cdots + k_u $ by $p_u $ for brevity.

\begin{proof}
It is enough to show that, for any $h \leq q$, the pullback $\mathcal{M}(x_h)^* (\gamma_{j_1 \cdots j_t})$ is $\gamma_{j_1 \cdots j_t}$ itself. Notice, by definition, the following pullback formula:
\[\mathcal{M}(x_h)^* (b_{j_1 \cdots j_k})= \beta_{j_1 \cdots j_k} + \delta_{j_k,h}\beta_{j_1 \cdots j_{k-1}} ,\]
\[\mathcal{M}(x_h)^* (d X_{j_1 \cdots j_k}) = d X_{j_1 \cdots j_k} + \delta_{j_k, h}d X_{j_1 \cdots j_{k- 1}}. \]
Thus, the pullback $\mathcal{M}(x_h)^* (\prod_{w= 1}^{ u} \beta_{j_{1+p_{w-1}}j_{2+p_{w-1}} \cdots j_{p_{w}}} d X_{j_1 \cdots j_{k_0}} )$ is formed as
$$ (\prod_{w: \ 1\leq w \leq u} \beta_{j_{1+p_{w-1}} \cdots j_{ p_{w}}} + \delta_{h, j_w} \beta_{j_{1+p_{w-1}} \cdots j_{p_{w} -1}}) (d X_{j_1 \cdots j_{k_0}} + \delta_{h, j_{k_0}} d X_{j_1 \cdots j_{ k_0 -1}}). $$
Denote the coefficients of $d X_{j_1 \cdots j_{k_0}}$ and of $d X_{j_1 \cdots j_{k_0 -1}}$ by $A_{k_0}$ and $B_{k_0}$, respectively. Then, by a careful observation, we can check that
$$(-1)^{u} \sum_{(k_1, \dots, k_u) \in \mathcal{S}_{t,k_0}} (A_{k_0} -B_{k_0+1}) = (-1)^{u} \sum_{(k_1, \dots, k_u) \in \mathcal{S}_{t,k_0}} \beta_{j_{1+p_0} \cdots j_{p_1}} \beta_{j_{1+p_1} \cdots j_{p_2}}
\cdots \beta_{j_{1+ p_{t-1}} \cdots p_t}. $$
Since the sum of the left hand side running over $1 \leq k_0 \leq t$ is equal to $\mathcal{M}(x_h)^* (\gamma_{j_1 \cdots j_t}) $, we have the desired $ \mathcal{M}(x_h)^* (\gamma_{j_1 \cdots j_t}) = \gamma_{j_1 \cdots j_t} $.
\end{proof}

\begin{lem}\label{ma3188}
Fix an index $(j_1 , \dots, j_k)$. Then, for any $s<t \leq k$ with $(s,t)\neq (1,k)$, we have
$$ d \gamma_{j_{s} \cdots j_{t}} = \sum_{r: \ s \leq r \leq t-1} \gamma_{j_{s} \cdots j_{r}} \wedge \gamma_{j_{r+1} \cdots j_{t}} \in \wedge^2 \mathrm{Im} (\mathcal{M}_{\R}) . $$
\end{lem}
\begin{proof}
First, by the Leibniz rule, the left hand side $d \gamma_{j_{s} \cdots j_{t}} $ is expressed as
$$ \sum_{k_0=s}^t ( \sum_{(k_1, \dots, k_{u} )\in \mathcal{S}_{t,k_0}} \sum_{v=1}^u (-1)^{u} \beta_{j_{1+p_0} \cdots j_{p_1}} \cdots
\check{\beta}_{j_{1+ p_v} \cdots j_{p_{v+1}}} \cdots \beta_{1+ j_{p_{u-1}} \cdots j_u}
d X_{j_{1+p_u} \cdots j_{p_{u+1}}}) \wedge d X_{j_s \cdots j_{k_0}}. $$
Here, the check $\check{\beta}_{j_{p_v+1} \cdots j_{p_{v+1}}} $ means the elimination of the term $\beta_{j_{p_v+1} \cdots j_{p_{v+1}}} $. On the other hand, the right hand side becomes
\[ \sum_{r: \ s \leq r < t} \Bigl(
\sum_{k_0' : \ s \leq k_0' \leq r} ( \sum_{(k_1', \dots, k_{u-1}')\in \mathcal{S}_{r ,k_0'}} (-1)^{u'} \beta_{j_{1+s} \cdots j_{p_1'}} \cdots \beta_{j_{1+ p_{u-1}'} \cdots j_r}
) d X_{j_s \cdots j_{k_0'}} \]
\[ \ \ \ \ \ \ \ \ \ \ \ \wedge \sum_{k_0'' : \ r+1 \leq k_0''\leq t} ( \sum_{(k_1'', \dots, k_u'') \in \mathcal{S}_{t-r ,k_0''}} (-1)^{u''} \beta_{j_{1+ k_0''} \cdots j_{p_1''}} \cdots \beta_{1 + j_{p_{u-1}''} \cdots j_t}
) d X_{j_{r+1} \cdots j_{k_0''}} \Bigr) . \]
By replacing $r$ by $k_0$, $k_a' $ by $k_{a+1} $ and $k_a''$ by $k_{u+a+1}$, a careful comparison deduces that this sum equal to the preceding expansion of the left hand side.
\end{proof}

This situation is the same as the defining system, as mentioned in \S \ref{Sbra1}. Note that the de Rham theorem preserves the cup product. Thus, in parallel to the main theorem \ref{maainthm}, we readily obtain the basis of the 2-cocycle of $H^2_{\rm dR}( BF/F_k)$ as follows.

\begin{thm}\label{32az}
The second cohomology $H^2_{\rm dR}( BF/F_k) \cong \R^{N_k}$ is spanned by the Massey products $ \langle \alpha_{i_1}, \dots, \alpha_{i_k} \rangle $ running over standard sequences $(i_1 \cdots i_k ) \in \mathfrak{U}_k$. Here, $ \langle \alpha_{i_1}, \dots, \alpha_{i_k} \rangle $ is represented by the 2-form
$$\sum_{k_0 =1}^k ( \sum_{(k_1, \dots, k_u) \in \mathcal{S}_{k ,k_0}}\sum_{v=1}^u (-1)^{k_0} \beta_{j_{1+ k_0} \cdots j_{p_1}} \cdots
\check{\beta}_{j_{1+ p_v} \cdots j_{p_{v+1}}} \cdots \beta_{j_{1+ p_{u-1}} \cdots j_t}
) d X_{j_{1+p_u} \cdots j_{p_{u+1}}}\wedge d X_{j_1 \cdots j_{k_0}}.
$$
\end{thm}

\begin{proof}
This resulting is immediately computed as the sum $\sum_{r=1}^{ t} \gamma_{j_{1} \cdots j_{r}} \wedge \gamma_{j_{r+1} \cdots j_{t}} $ by the definition of the Massey product. Here, we should remark that the sum is $F/F_k$-invariant by Lemma \ref{maas22}.
\end{proof}

As a result for 3-cocycles in $H^3_{\rm dR}(B(F/F_k))$, Theorem \ref{maainthm} (III) and Proposition \ref{maainthm22} enable us to similarly describe 3-cocycles $\mathfrak{s}_{\ell}( dX_j \wedge \langle \alpha_{i_1}, \dots, \alpha_{i_{\ell}} \rangle ) $ as 3-forms, where $ \ell =k,k+1$. Day \cite{Day} considers an extension of the Morita homomorphism
from differential 3-forms of $B (F/F_k) $; thus, it seems interesting to observe the work from the viewpoint of Theorem \ref{32az}. 

Finally, we conclude this appendix by describing some examples with $ t \leq 4$.
\begin{exa}\label{m3231}
\begin{enumerate}[(i)]
\item The 1-form $\gamma_{ab}$ is $dX_{ab} - \beta_a dX_b $. Hence, the Massey product $ \langle \alpha_a, \alpha_b, \alpha_c \rangle =\gamma_{ab} \wedge \gamma_c + \gamma_{a} \wedge \gamma_{bc}$ is expressed as $dX_{a} \wedge d X_{bc} + dX_{ab} \wedge d X_c - \beta_a dX_b d X_c - \beta_b d X_a dX_c. $
\item Next, when $t=3$, the 1-form $\gamma_{abc}$ is $ dX_{abc} - \beta_c dX_{ab} - \beta_{b}\beta_{c} dX_{a}+ \beta_{bc} dX_{a}$. Hence, the Massey product $ \langle \alpha_a, \alpha_b, \alpha_c, \alpha_d \rangle $ is formulated as
\[ (dX_{abc} - \beta_c dX_{ab} - \beta_{b}\beta_{c} dX_{a}+ \beta_{bc} dX_{a})\wedge dX_d + (dX_{ab} - \beta_a dX_b)\wedge (dX_{cd} - \beta_c dX_d) \]
$$ \ \ \ \ \ \ \ \ + \beta_a \wedge(dX_{bcd} - \beta_d dX_{bc} - \beta_{c}\beta_{d} dX_{b}+ \beta_{cd} dX_{b}). $$
\item Next, when $t=4$ and $(j_1,j_2,j_3,j_4)=(a,b,c,d)$, the 1-form $\gamma_{abcd}$ is
$$ dX_{abcd} - \beta_d dX_{abc} + \beta_{c}\beta_{d} dX_{ab}+ \beta_{cd} dX_{ab}
- \beta_{bcd} dX_{a} -\beta_{bc}\beta_{d} dX_{a}- \beta_b \beta_{cd} dX_{a}- \beta_b \beta_{c} \beta_{d} dX_{a} . $$
\noindent
Then, $ \langle \alpha_a, \alpha_b, \alpha_c, \alpha_d , \alpha_e \rangle $ can be similarly computed as $\gamma_{abcd}\wedge  \gamma_e + \gamma_{abc}\wedge  \gamma_{de} +\gamma_{ab} \wedge \gamma_{cde} +\gamma_{a}\wedge  \gamma_{bcde}$.
\end{enumerate}
\end{exa}

\vskip 1pc

\normalsize
DEPARTMENT OF
MATHEMATICS
TOKYO
INSTITUTE OF
TECHNOLOGY
2-12-1
OOKAYAMA
, MEGURO-KU TOKYO
152-8551 JAPAN

\end{document}